\documentclass[12pt]{amsart}
\usepackage{pictexwd,dcpic}
\usepackage{amsmath,amsfonts,amssymb,amscd,amsthm,latexsym,mathrsfs}

\usepackage{pstricks,pst-node,pst-plot}
\usepackage{graphicx}
\usepackage{pstricks,pst-node,pst-plot}
\usepackage{pictexwd,dcpic}
\usepackage{amsmath,amscd}
\usepackage{amsmath}
\usepackage{amsfonts}
\usepackage{amsthm}

\theoremstyle{plain}
\newtheorem{thm}{Theorem}[section]
\newtheorem{cor}[thm]{Corollary}
\newtheorem{lem}[thm]{Lemma}

\theoremstyle{definition}
\newtheorem{defi}[thm]{Definition}
\newtheorem{rem}[thm]{Remark}

\begin{document}

\title[Some more axiomatisability for $S$-acts
 ]{Some more axiomatisability  for $S$-acts}
\subjclass{20 M 30, 03 C 60}
\keywords{axiomatisability,flat, $S$-acts}
\date{\today}

\author{Lubna Shaheen}
\email{lls502@york.ac.uk}
\address{Department of Mathematics\\University
  of York\\Heslington\\York YO10 5DD\\UK}

\begin{abstract} This paper discusses necessary and sufficient conditions on a monoid $S$, such that the class of $\mathcal{C} $-flat left $S$-acts is axiomatisable, where $\mathcal{C}$ is the class of all embeddings (of right ideals into $S$) of right $S$-acts. We consider the axiomatisability of some flatness classes of $S$-acts, which were previously discussed by Bulman-Fleming and Gould \cite{bulmanfleminggould} . We present here a more general procedure to axiomatise these classes.
A similar type of general results have been found  for $S$-posets by  Gould and Shaheen
\cite{gouldshaheen1}.
We have found that there are some classes of
$S$-acts which are axiomatisable by more than one method. This has not been seen before.

\end{abstract}

\maketitle
\section{Introduction}\label{sec:intro}

In this paper, we will be considering axiomatisability 
problems for classes of (left) $S$-acts. Let $L_{S}$ be the first order language relating to left $S$-acts.  Then the class $S$-Act is axiomatised by
$\Sigma_S=\{ (\forall x)(1x=x)\}
\bigcup \{ \varphi_{s,t}:s,t\in S\}$
where 
\[\varphi_{s,\,t}:= (\forall
x)(s(t(x))={st}(x)).\]

 It can be noted that there are certain  classes of $S$-acts that
are axiomatisable for all monoids $S$, e.g. the class {\em $S$-Act} of all left $S$-acts. Less trivially, we denote by $\mathcal{T}$ the class of torsion free
left $S$-acts. A left $S$-act $A$ is {\em torsion free} if $sa=sb\mbox{ implies that }a=b$ for all $s\in \mathcal{LC}$, where $\mathcal{LC}$ denotes the set of left cancellable elements of
$S$. Clearly $\mathcal{T}$ is axiomatised by
$\Sigma_S\cup \big\{ (\forall x,y)(sx=sy \rightarrow x=y): s\in\mathcal{LC}\big\}.$
However,  there are some classes of 
 $S$-acts which are axiomatisable for some monoids 
and not for others e.g. the classes $\mathcal{SF}$ and $\mathcal{P}r$ 
of strongly flat and projective
$S$-acts
are axiomatisable if $S$ is finite or a group, but 
for the monoid $C$ where $C=\{ 1= e_0,e_1,e_2,\cdots \}$ and 
$e_i e_j = e_{max{\{i,j}\}} $, that is, $C$ is an inverse $\omega$-chain,  the class $\mathcal{SF}$
 is axiomatisable but $\mathcal{P}r$ is not \cite{gould}.
 
 Introductory work on axiomatisability problems for $S$-acts was done by Gould \cite{gould}. She considered the following questions:
 for which monoids $S$ are the classes of $\mathcal{SF}$ and $\mathcal{P}r$ axiomatisable?  She described necessary and sufficient conditions on
 $S$ such that $\mathcal{SF}$
is axiomatisable and obtained partial results for $\mathcal{P}r$. The
full answer for $\mathcal{P}r$ was  provided by Stepanova 
 \cite{step}.
The kind of conditions that arise, here as for other questions, are finitary in nature.

Later  Bulman-Fleming and  Gould \cite{bulmanfleminggould} gave an alternative proof of 
Stepanova's result of axiomatisability of projective $S$-acts. They also characterised 
those monoids such that the classes $\mathcal{F}$(flat) and $\mathcal{WF}$(weakly flat) of  $S$-acts are axiomatisable. Subsequently, Gould \cite{gould:tartu} characterised those monoids $S$ such that the class $\mathcal{F}r$(free)  
$S$-acts were axiomatisable. In \cite{gould:tartu} there is a discussion  of  the relations between the conditions on a monoid $S$ that arise while axiomatising certain classes of $S$-acts such as $\mathcal{F}r,\,\mathcal{P}r,\,\mathcal{SF},\,\mathcal{F}$ or $\mathcal{WF}$. Recently Gould, Stepanova, Mikhalev and 
Palyutin \cite{gouldpalyutin} gave a comrehensive survey named ``Model Theoretic Properties of Free,
Projective and Flat S-acts'' which includes much additional model theoretic material.

The aim of this paper is to  add the theory of 
axiomatisability of classes of $S$-acts  over a monoid $S$. We put some of the techniques of
earlier articles into a general setting. In \cite{gouldshaheen1} we  use these methods to develop the theory of 
axiomatisability of $S$-posets over a pomonoid $S$.

It is known that there are three familiar methods to axiomatise classes of $S$-acts. The first of them is the simplest making use of interpolation conditions  on $S$-acts to produce finitary conditions on $S$. This method has been used by Gould for $\mathcal{SF}$ \cite{gould}; we will  refer to this as the `` elements '' method. We have used this in the context of, for example, Condition (EP),(W), and (PWP), for $S$-acts.

The second two methods both involve  ``replacement tossings'' and have been developed by Bulman-Fleming and Gould  in \cite{bulmanfleminggould} for $\mathcal{F}$ and $\mathcal{WF}$; we will refer to these as ``replacement tossings'' methods; we have used these in the perspective of, Condition (E),(P),(EP),(W) and (PWP), for $S$-acts.

First, we consider the axiomatisability of some classes of $S$-acts related to flatness, such as $\mathcal{F},\mathcal{WF}$ and $\mathcal{PWF}$(principally weakly flat), where first two are previously discussed by Bulman-Fleming and  Gould \cite{bulmanfleminggould}. In Section~\ref{sec:gaxsacts} we demonstrate a more general way to axiomatise these classes, putting the two of the  ``replacement tossings'' methods into an abstract context. These can then be specialised to prove both new and known results. For clarity we are giving  general results without proves which can be solved on similar lines as for $S$-posets case, see \cite{gouldshaheen1}. Interested reader may find proves in \cite{shaheen:2010}.

In Section ~\ref{sec:axsacts}, we investigate the axiomatisability of the classes  $\mathcal{EP},\,\mathcal{W}$ and $\mathcal{PWP}$.  We determine
when these classes of $S$-acts are axiomatisable by using both the ``elements'' method and by using ``replacement tossings''.

In Section ~\ref{sec:examples} we attempt  some examples  of axiomatisability.  We develop the connection between axiomatisability  conditions of different classes. We know that if $\mathcal{P}$ is axiomatisable then so is $\mathcal{W}$. We give an example of a monoid such that $\mathcal{W}$ is axiomatisable but $\mathcal{P}$ is not. It is known that  $\mathcal{E}$ implies $ \mathcal{EP}$, we would like to know whether $\mathcal{EP}$ is axiomatisable if $\mathcal{E}$ is axiomatisable but this  is still unknown.

We note that if Condition A implies Condition B, where A and B are conditions on left $S$-acts, then we usually expect that if the class $\mathcal{A}$ of left $S$-acts satisfying Condition A is axiomatisable, then so is the class $\mathcal{B}$ of the left $S$-acts satisfying Condition B.

\begin{lem}\label{lem:UV} Let $S$ be a monoid, and let $\mathcal{U},\mathcal{V}$ be classes of left $S$-acts such that $S \in \mathcal{U}$, and  $\mathcal{U} \subseteq \mathcal{V}$. Suppose $\mathcal{V}$ is axiomatisable if and only if every ultrapower of $S$ lies in $\mathcal{V}$. Now if $\mathcal{U}$ is axiomatisable then so is $\mathcal{V}$.
\end{lem}

 Surprisingly,  we have managed to show without using Lemma ~\ref{lem:UV}, which has been extensively used throughout this paper,  that if $\mathcal{P}$ is axiomatisable then so is $\mathcal{EP}$.

\section{Definitions and Preliminaries}

Let $A$ be a non-empty set and let $S$  be a monoid, and suppose 
there is a function $S\times A\rightarrow A$, where
$(s,a) \mapsto sa$ with the following properties $ s(t(a)) = (st)a$  and $ 1  a = a$ for all $ s , t \in S$ and $ a \in A$
is said to be a {\em left $S$-act}.
The notion of {\em right $S$-act} is defined dually.
 The class of all left $S$-acts
are denoted by $S$-{\em Act}. Notice
that
$S$ may be regarded as both a left and a right $S$-act, with actions
given by the binary operation in $S$.

A subset $B$ of a left $S$-act $A$ is an  ($S$-){\em subact}
if $B$ is closed under the action of $S$. Any left ideal of $S$ is a
subact of $S$ and dually, any right ideal is a subact of $S$.

Let $A$ be a left $S$-act and $\rho$ a relation on $A$. Then
$\rho$ is a {\em (left) $S$-act congruence} if $\rho$ is an equivalence
relation such that for any $a,b\in A$ and $s\in S$, if $a\,\rho\, b$,
then $sa\,\rho\, sb$.

A function $\theta:A \to B$ from a left $S$-act $A$ to a left $S$-act $B$ is
called  an {\em $S$-morphism} if $ (sa)\theta = s(a)\theta  $ 
for all $ s \in S$ and $a  \in A$. A bijective $S$-morphism is called
an $S$-{\em isomorphism}; if there exists an $S$-isomorphism from $A$ to
$B$, then we say that $A$ and $B$ are {\em isomorphic} and write
$A\cong B$.
We will denote by $\mathbf{S}$-{\bf Act} the category
 with objects all left $S$-acts and morphisms the $S$-morphisms between 
them. Dually we can define category with objects all right $S$-acts and
morphisms
 as $S$-morphisms between them, and will denote it by  {\bf
Act}-$\mathbf{S}.$

The approach to concepts of flatness is rather more complicated, and involves
the notion of tensor product, which we now describe.

Let $A$ be a right $S$-act and $B$ be a left $S$-act.
The {\em tensor product} of $A$ and $B$ is  obtained by taking the
quotient of $A \times B$ by the equivalence relation
 generated by the set $\{\big((as,b),(a,sb)\big)\mid a \in A, b \in B, s \in S
\}$.
We will use $A \otimes B$  to denote the tensor product of
$S$-acts $A$ and $B$. The equivalence class of $(a,b) \in  A \times B$
will be denoted by $a \otimes b \in A \otimes B$.

\bigskip

We will need to look 
carefully at equalities of the form $a\otimes b=a'\otimes b'$.

\begin{lem} \label{lem:tossing}\cite{kilpknauer}
Let $A$ be a right $S$-act and $B$ a left $S$-act.
Then for $a,a'\in A$ and $b,b'\in B$,
$a\otimes b=a'\otimes b'$ if and only if there exist
$s_1,t_1,s_2,t_2,\hdots ,s_m,t_m\in S$, $a_2,\hdots ,a_m\in A$
and $b_1,\hdots ,b_m\in B$ such that
\[\begin{array}{rclcrcl}
  &    &                    & b& = & s_1 b_1\\
                 a s_1 & = & a_2 t_1        & t_1 b_1 & =& s_2 b_2\\   
                a_2 s_2 & = & a_3 t_2     &  t_2 b_2 & = & s_3 b_3 \\
              &  \vdots &                         && \vdots & \\                    
      a_{m-1} s_{m-1} & = & a_m t_{m-1}    &  t_{m-1}b_{m-1} & = & s_m b_m  \\
              a_m s_m &=& a' t_m          & t_m b_m & = & b' 
\end {array} \]
\end{lem}

The sequence presented  in Lemma \ref{lem:tossing} will be called a {\em tossing}
(or scheme) $\mathcal{T}$
 of length $m$ over $A$ and $B$ connecting
$(a,b)$ to $(a',b')$. The {\em skeleton}
$\mathcal{S}=\mathcal{S}(\mathcal{T})$ of $\mathcal{T}$,
 is the sequence $\mathcal{S}=(s_1,t_1,\hdots ,s_m,t_m)\in S^{2m}.$
The set of all skeletons is denoted by $\mathbb{S}$. By considering
trivial acts it is easy to
see that $\mathbb{S}$ consists of all even length
sequences of elements of $S$.


Now we define for an $S$-act $B$ the functor $-\otimes B:$
{\bf Act-}${\mathbf S} \to \mathbf{Set}$,
where $\mathbf{Set}$ is the category of sets and maps, by 



\vspace{5mm}
\[\mbox{
\begindc{0}[50]
\obj(0,2){$A$}
\obj(2,2){$A'$}
\obj(0,0){$A \otimes B$}
\obj(2,0){$A' \otimes B$}
\mor{$A$}{$A'$}{$f$}
\mor{$A$}{$A \otimes B$}{$- \otimes B$}
\mor{$A'$}{$A' \otimes B$}{$- \otimes B$}
\mor{$A \otimes B$}{$A' \otimes B$}{$ f \otimes I_{B}$}
\enddc}\]


\noindent where $f \otimes I_{B}:A \otimes B \to A' \otimes B$, and for $a \,\otimes b \in A \otimes B$, $$(a \, \otimes b)(f \otimes I_B)= (a)f\otimes b$$ where we have $f:A \to A'$ an $S$-morphism in {\bf Act-}${\mathbf S}$. 
Now we will see that various notions of flatness can be
drawn from this functor and involve it preserving monomorphisms,
or related concepts such as pullbacks and equalisers.

\bigskip

A left $S$-act $B$ is called {\em strongly flat}
 if the functor $-\otimes B $ preserves  pullbacks and  equalizers;
by a result of Stenstr\"{o}m  \cite{stenstrom}, $B$ is strongly flat if and
only if $B$
satisfies interpolation conditions later called (P) and (E) which are defined as follows:

$(P)$: for all $b,b' \in B$ and $s,s' \in S$ if $sb = s'b'$ then there exists $b'' \in B$ and $u,u' \in S$  such that $b=u b'', b'= u'b''$ and $ su = s'u'$,

$(E)$: for all $b \in B$ and $s,s' \in S$ if $sb = s'b$ then there exists $b'' \in B$ and $u \in S$ such that $b= ub''$ and $su = s'u$.

A left $S$-act $A$ satisfies Condition (EP) if whenever $ s \, a = t \, a $ 
for some $s,t\in S$ and $a\in A$,
then there exists $ a'' \in A , \, u, \,v\,\in S$ such that $ a =u\, a''= v\,a''   $  with $ s \,u = t
\,v.$

 We will denote the class of left $S$-acts satisfy Condition $(EP)$,
$(P)$ and $(E)$ by $\mathcal{EP}, \, \mathcal{P},\,\mathcal{E}$ respectively.

 A left $S$-act $B$ is  {\em flat}
 if it preserves  embeddings of right $S$-act, which is easily seen to
be equivalent to the following: if we have $a \otimes b = a' \otimes b' $ in
$ A \otimes B$ then the equality also holds in $ (aS \cup a'S )\otimes B$
for all $ a , a' \in A $ and $ b, b' \in B$, and called  ({\em principally}){\em weakly
flat} if it preserves  embeddings of (principal)right ideals of $S$ into $S$,
or in other words if $ m \otimes b = m' \otimes b'$ in $ S_S
\otimes B$ then it is already exists in $K_S \otimes B$ where $K_S$
is any (principal)right ideal of $S_S$ for all $ m, m' \in K_S$ and $b,b' \in B$.

We will denote the classes of strongly flat, flat, weakly flat, and  prinipally weakly flat, by $\mathcal{SF},\,\mathcal{F},\,\mathcal{WF},\,\mathcal{PWF}$ respectively.

Unlike the case for strongly flat there are no simple conditions 
such as $(P)$ and $(E)$ in the flat or weakly flat or principally weakly
flat cases.

In \cite{kilpknauer}  Bulman-Fleming and McDowell has proved that a left $S$-act $A$ is weakly flat if and only if it is principally weakly flat and satisfies Condition;

$(W)$: If $sa = ta'$ for $a, a' \in A$, $ s, t \in S$ then there exists $ a'' \in A$ $u \in sS \cap tS$ such that $sa=ta'=ua''$, where we can visualize $u$ as  $u=ss'=tt'$ for some $s',t' \in S$. 

\smallskip
We will denote the class of all left $S$-acts satisfy Condition $(W)$ by $\mathcal{W}$.

\begin{rem} \cite{kilpknauer}
In ${\mathbf S}${\bf -{Act}} we have
\[ \mathcal{SF} \Rightarrow \mathcal{F}
 \Rightarrow \mathcal{WF}\Rightarrow \mathcal{PWF}\]
\end{rem}

We refer reader to \cite{shaheen:2010} for definitions of pullback  diagram
$(P,(p_1,p_2))$ of the pair $(f_1,f_2)$ and related concepts of
pullback flat $S$-act.

A left $S$-act satisfies condition $(PWP)$ if for every pullback diagram $(P,(p_1,p_2))$ of the pair $(f_1,f_2)$ where $f_i : sS \to S$, $i : 1,2$ the corresponding map $\gamma$ is surjective.

\noindent Equivalently \cite{laan}, a left $S$-act satisfies condition $(PWP)$ if

\bigskip

$\forall \,\,a,\, a' \in A ,\,\forall \,\,t \in S ,  ta = ta'  \,\Rightarrow\, \exists \, a'' \in A, \, \, u, v \in S$ such that

 $$  a = u a'' \wedge a' = v a'' \wedge tu = tv.$$  
 
\noindent The class of left $S$-acts satisfy Condition $(PWP)$ will be denoted by
$\mathcal{PWP}$.
\bigskip

\noindent We now come to the main theme of this paper.

 Suppose you want to construct a  first order language which will axiomatise
a class of a given type of algebras say $\mathcal{C}$. We can do this if  
we succeed in 
defining a suitable set of sentences $\Sigma$ in our 
first order language such that any algebra $A$ lies in $\mathcal{C}$
if and only if all of the
sentences of $\Sigma$ are true in $A$.
Investigating axiomatisability problems of this kind  
involves some basics of model theory, in particular, of ultraproducts. For left $S$-acts the first order language has no constant or
relational symbols (other than $=$) and consists of 
a unary function symbol say $\lambda_{s}$ for each $s \in S$.
 We denote the first order language relating to left $S$-acts by $L_{S}$.
We usually write $\lambda_s(x)$ as $sx$. Then the class $S$-Act is axiomatised by
\[\Sigma_S=\{ (\forall x)(\lambda_1(x)=x)\}
\bigcup \{ \varphi_{s,t}:s,t\in S\}\]
where 
\[\varphi_{s,t}:= (\forall
x)(\lambda_s(\lambda_t(x))=\lambda_{st}(x)).\]

The next result play an important role in question of axiomatisability.
 
\begin{thm}{\em(\L os's Theorem)}\label{thm:los}\cite{ck}
Let $L$ be a first order language, and let $\mathcal{C}$ be
a class of interpretations of $L$. If $\mathcal{C}$ is
axiomatisable, then $\mathcal{C}$ is closed under ultraproducts.

\end{thm}

\section{General results on axiomatisability}\label{sec:gaxsacts}

Let $\mathcal{C}$ be a class of embeddings of right $S$-acts, for example, all embeddings, or all embeddings of right ideals into $S$. A left $S$-act $B$ is called $\mathcal{C}$-flat if the functor $- \otimes B$ maps embeddings in $\mathcal{C}$ to one-one maps in {\bf Set}, that is, if $\tau: A \to A'$ is in $\mathcal{C}$, then $ \tau \otimes I_{B}$ is one-one. In terms of elements this says that if $ a,a' \in A$ and $b,b' \in B$ and $a \tau \otimes b= a' \tau \otimes b'$ in $A' \otimes B$, then $ a \otimes b= a' \otimes b'$ in $A \otimes B$. We denote the class of $\mathcal{C}$-flat left $S$-act by $\mathcal{CF}$.
Note: for $S$-posets there will be two variations of the notion of $\mathcal{C}$-flat,
as we explain in \cite{gouldshaheen1}. 

We introduce Condition (Free) on $\mathcal{C}$ below.
 In Subsection~\ref{subsec:general2a}  we  find necessary and sufficient conditions for the class $\mathcal{CF} $ of $\mathcal{C}$-flat left $S$-acts to be axiomatisable if
 $\mathcal{C}$  satisfies Condition (Free). The result of Bulman-Fleming and Gould axiomatising $\mathcal{F}$
 becomes a special case. In Subsection~\ref{subsec:sactsoutFP}  we drop the assumption of Condition (Free). We have a general result  to determine for which monoids $S$ is $\mathcal{CF}$ axiomatisable. 
The result of Bulman-Fleming and Gould axiomatising $\mathcal{WF}$ then
 becomes a special case. We can also deduce the axiomatisability result for $\mathcal{PWF}$ using this method.

We first describe our two general results involving ``replacement tossings''. Some of the arguments are rather intricate. The reader wanting an easier introduction to axiomatisability problems could look at Section~\ref{sec:axsacts} first.

\subsection{Axiomatisability of $\mathcal{CF}$ with Condition (Free)}\label{subsec:general2a}

In this subsection we find necessary and sufficient conditions on $S$ such that a class
$\mathcal{CF}$ is axiomatisable, where $\mathcal{C}$ is a class of embeddings of
left $S$-acts satisfying Condition (Free). We first describe this condition.

It is convenient to introduce some notation. Let 
$$\mathcal{S}=(s_1,t_1,\cdots,s_n,t_n) \in \mathbb{S}$$
\noindent be a skeleton. Let $R_{S}$ be the first order language relating to right $S$-acts.

We define a formula $\epsilon_{\mathcal{S}} \in R_{S}$,  as follows:
$$\epsilon_{\mathcal{S}}(x,x_2,\cdots,x_n,x'):= xs_1=x_2t_1\wedge x_2 s_2=x_3 t_2\wedge \cdots \wedge x_n s_n = x't_n$$ and put $$ \delta_{\mathcal{S}}(x,x'):=(\exists x_2\cdots\exists x_n){ \epsilon_{\mathcal{S}}}(x,x_2,\cdots,x_n,x').$$

On the other hand we define the formula
$$ \theta_{\mathcal{S}}(x , x_1,\cdots,x_n,x'):=x=s_1 x_1 \wedge t_1 x_1 = s_2 t_2 \wedge \cdots \wedge t_n b_n =x' $$
\noindent of $L_{S}$ and put $$\gamma_{\mathcal{S}}(x,x'):= (\exists x_1\cdots\exists x_n)\theta_{\mathcal{S}}(x,x_1,\cdots,x_n,x').$$

\begin{defi}\label{defi:Cfree}
We say that $\mathcal{C}$ satisfies {\em Condition (Free)} if for each $\mathcal{S} \in \mathbb{S}$ there is an embedding $\tau_{\mathcal{S}}:W_{\mathcal{S}} \to W'_{\mathcal{S}}$ in $\mathcal{C}$ and $u_{\mathcal{S}},u'_{\mathcal{S}} \in W_{\mathcal{S}}$ such that $\delta_{\mathcal{S}}( u_{\mathcal{S}} \tau_{\mathcal{S}},u'_{\mathcal{S}} \tau_{\mathcal{S}})$ is true in $W'_{\mathcal{S}}$ and further, for any embedding $ \mu:A \to A'$ in $\mathcal{C}$ and any $ a, a' \in A$ such that $\delta_{\mathcal{S}}(a \mu , a' \mu)$ is true in $A'$, there is a morphism $ \nu:W'_{\mathcal{S}} \to A'$ such that $ u_{\mathcal{S}} \tau_{\mathcal{S}} \nu = a \mu ,\,u'_{\mathcal{S}} \tau_{\mathcal{S}}\nu = a' \mu$ and $W_{\mathcal{S}} \tau_{\mathcal{S}} \nu \subseteq A \mu$.
\end{defi}

\begin{rem}\label{rem:formulas} Let $A,B$ be right and left $S$-acts, respectively,
let $a,a'\in A$ and $b,b'\in B$. 

$(i)$ The pair $(a,b)$ is connected to the pair $(a',b')$
via a tossing with skeleton $\mathcal{S}$ if and only if
$\delta_{\mathcal{S}}(a,a')$ is true in $A$ and
$\gamma_{\mathcal{S}}(b,b')$ is true in $B$.

$(ii)$ If $\delta_{\mathcal{S}}(a,a')$  is true in $A$ and $\psi:A\rightarrow A'$ is
a (right) $S$-morphism, then $\delta_{\mathcal{S}}(a\psi,a'\psi)$ is true
in $A'$.

$(iii)$ If   $\gamma_{\mathcal{S}}(b,b')$ is true in $B$ and $\tau:B\rightarrow B'$ is
(left) $S$-morphism, then $\gamma_{\mathcal{S}}(b\tau,b'\tau)$ is true in
$B\tau$.
\end{rem}


\begin{lem}\label{lem:cfsacts}
Let $\mathcal{C}$ be a class of embeddings of right $S$-acts satisfying {\em Condition (Free)}. Then the following conditions are equivalent for a left $S$-act $B$:

(1) B is $\mathcal{C}$-flat;

(2) $- \otimes B$ preserves all embeddings $\nu_{\mathcal{S}}:W_{\mathcal{S}} \to W'_{\mathcal{S}}$;

(3) if $(u_{\mathcal{S}} \tau_{\mathcal{S}},b) $ and $(u'_{\mathcal{S}}\tau_{\mathcal{S}},b' )$ are connected by a tossing over $W'_{\mathcal{S}}$ and $B$ with skeleton $\mathcal{S}$, then $(u_{\mathcal{S}} ,b)$ and $(u'_{\mathcal{S}},b')$ are connected by a tossing over $W_{\mathcal{S}}$ and $B$.
\end{lem}
\begin{proof}

Proof is along the similar lines as given for ordered case in \cite{gouldshaheen1}, or see \cite{shaheen:2010}.
\end{proof}

We use  ``The Finitely Presented Flatness Lemma'' \cite{bulmanfleminggould} for $S$-acts to construct an example of the use of Condition (Free). Specifically, we show that the class of all right $S$-acts has
Condition (Free).

 Let $\mathcal{S}=(s_1,t_1,\cdots,s_m,t_m)$ be a skeleton and let  $F^{m+1}$ be the free right $S$-act 
 $$ xS \,\dot{\cup}\, x_2 S \,\dot{\cup} \cdots \dot{\cup} \,x_m S \, \dot{\cup} \, x' S.$$  
 Let $ \rho_{\mathcal{S}}$ be the $S$-act congruence on $F^{m+1}$ generated by the relation $R_{\mathcal{S}}$ $$\{(xs_1,x_2 t_1 ),(x_2 s_2, x_3 t_2 ), \cdots, (x_{m-1} s_{m-1},x_m t_{m-1}),(x_m s_m, x' t_m ) \}.$$
We denote the $\rho_{\mathcal{S}}$-class of $w \in F^{m+1}$ by $[w]$. 

\noindent If $B$ is a left $S$-act and $ b, b_1, \cdots, b_m, b' \in B$ are such that $$ b= s_1 b_1, \, t_1 b_1 = s_2 b_2 , \cdots, t_m b_m = b' $$
then the tossing

$$
\begin{array}{rclcrcl}
                      &    &                    & b& = & s_1 b_1\\\
                [x]s_1 & = & [x_2] t_1        & t_1 b_1 & = & s_2 b_2\\\
                                
                [x_2] s_2 & = & [x_3] t_2     &  t_2 b_2 & = & s_3 b_3 \\
                
              &  \vdots &                         && \vdots &\\\
                              
             [x_{m-1}] s_{m-1} &= & [x_m]t_{m-1}& t_{m-1} b_{m-1} & =& s_{m}b_{m} \\\  
               [ x_m ]s_m & = & [x']t_m          & t_m b_m & = & b' \\

\end{array}
$$
 over $F^{m+1} / \rho_{\mathcal{S}}$ and $B$ is called the {\em standard tossing} with skeleton 

$$\mathcal{S}
= (s_1,t_1,\cdots, s_m,t_m )$$ of length $m$ connecting $([x],b)$ to $([x'],b' ).$

\begin{lem}\label{lem:fflat}\cite{bulmanfleminggould} The following conditions are  equivalent for a left $S$-act $B$:

(i) $B$ is flat;

(ii) $-\otimes B$ maps the embeddings of $[x]S \cup [x']S$ into $F^{m+1}/ \rho_{\mathcal S}$ in the category {\bf Act-}${\mathbf S}$ to monomorphisms in the category of ${\bf Set}$, for every skeleton $\mathcal S$;

(iii)  if $\big([x],b \big)$ and $\big([x'],b' \big)$ are connected by a standard tossing over $F^{m+1}/ \rho_{\mathcal S}$ and $B$ with skeleton  $\mathcal S$, then they are connected by a tossing over $[x]S \cup [x']S$ and $B$.
\end{lem}

We therefore able to show that:

\bigskip

\begin{lem}\label{lem:allfree} The class Act-$S$ of all right $S$-acts has
Condition (Free).

\begin{proof}
Let $\mathcal{S}$ be a skeleton of length $m$, and let  $W_{\mathcal{S}}=F^{m+1} / \rho_{\mathcal{S}}$, 
$W_{\mathcal{S}}'=[x]S\cup [x']S$, and let $\tau_{\mathcal{S}}:W_{\mathcal{S}} \to W^{'}_{\mathcal{S}}$ denote inclusion. Then for $[x],[x'] \in W_{\mathcal{S}}$, put
$u_{\mathcal{S}}=[x]$ and $u'_{\mathcal{S}}=[x']$. Clearly,
$\delta_{\mathcal{S}}(u_{\mathcal{S}}\tau_{\mathcal{S}},u'_{\mathcal{S}}\tau_{\mathcal{S}})$
is true in $W'_{\mathcal{S}}$.

Suppose that $\mu:A\rightarrow A'$ is any right $S$-act
embedding with $\delta_{\mathcal{S}}(a \mu,a' \mu)$ is true in $A^{'}$

$$
\begin{array}{rclcrcl}
                      &    &                    &  a \mu  s_1 & = & a_2  t_1\\\
                      &    &    & a_2 s_2& = & a_3 t_2\\\

              & &                         && \vdots &\\\
                            
               &  &           & a_m s_m & = &a'\mu t_m . \\

\end{array}
$$
Define $\psi:F^{m+1} \rightarrow A^{'}$ by $x \psi = a \mu $,\,\,$x_i  \psi= a_i \,\,\, 2 \leq i \leq m,\,\,x' \psi = a' \mu.$

Then $\rho_{\mathcal{S}} \subseteq \ker{\psi}$ so there exists  $\nu = \overline{\psi}:F^{m+1}/ \rho_{\mathcal{S}} \rightarrow A^{'}$, given by  $[k] \overline{\psi}= k \psi$. We have
\[u_{\mathcal{S}} \tau_{\mathcal{S}} \nu = [x] \overline{\psi} = x \psi = a \mu,\,\, 
u_{\mathcal{S}}' \tau_{\mathcal{S}} \nu = [x'] \overline{\psi} = x' \psi = a' \mu\]
so that 
\[W_{\mathcal{S}}\tau_{\mathcal{S}}\nu=([x]S\cup [x']S)\overline{\psi}=a\mu S\cup a'\mu S\subseteq A\mu.\]

\vspace{2mm}

\noindent Thus we can see that Condition (Free) holds.

\end{proof}

\end{lem}

Let $\mathcal{C}$ be a class of right $S$-acts, and let $\mathcal{\overline{C}}$ be the set of products of morphisms in $\mathcal{C}$ (with the obvious definition).

\begin{lem}\label{lem:CFree}
Let $\mathcal{C}$ be a class of embeddings of right $S$-acts, satisfying Condition {(Free)}. If a left $S$-act $B$ is $\mathcal{C}$-flat, then it is $\mathcal{\overline{C}}$-flat.
\end{lem}

We now come to our first main result. The technique used is that of \cite{bulmanfleminggould}, but we are working in a more general context.

\begin{thm}\label{thm:FoutP}
 Let $\mathcal{C}$ be a class of embeddings of right $S$-acts satisfying Condition (Free). Then the following conditions are equivalent for a monoid $S$:

 (1) the class $\mathcal{CF}$ is axiomatisable;
 
 (2) the class $\mathcal{CF}$ is closed under formation of ultraproducts;
 
 (3) for every skeleton $\mathcal{S} \in \mathbb{S}$ there exist finitely many replacement skeletons $\mathcal{S}_{1},\cdots,\mathcal{S}_{\alpha(\mathcal{S})}$ such that, for any embedding $\gamma: A \to A'$ in $\mathcal{C}$ and any $\mathcal{C}$-flat left $S$-act $B$, if $(a \gamma ,b),\, (a' \gamma , b') \in A' \times B$  are connected by a tossing  $\mathcal{T}$ over $A'$ and $B$ with $ \mathcal{S}(\mathcal{T})= \mathcal{S}$, then $(a,b)$ and $(a',b')$ are connected by a tossing ${\mathcal{T}}^{'}$ over $A$ and $B$ such that $\mathcal{S}({\mathcal{T}}^{'})= \mathcal{S}_{k}$, for some $ k \in \{1,\cdots,\alpha(\mathcal{S})\}$;
 
 (4) for every skeleton $\mathcal{S} \in \mathbb{S}$ there exist finitely many replacement skeletons $\mathcal{S}_{1},\cdots,\mathcal{S}_{\beta(\mathcal{S})}$ such that, for any $\mathcal{C}$-flat left $S$-act $B$, if $(u_{\mathcal{S}} \tau_{\mathcal{S}},b)$ and $(u'_{\mathcal{S}} \tau_{\mathcal{S}},b')$ are connected by a tossing $\mathcal{T}$ over $W'_{\mathcal{S}}$ and $B$ with $\mathcal{S}(\mathcal{T}) = \mathcal{S}$, then $(u_{\mathcal{S}}, b)$, and $(u'_{\mathcal{S}},b')$ are connected by a tossing ${\mathcal{T}}^{'}$ over $W_{\mathcal{S}}$ and $B$ such that $\mathcal{S}({\mathcal{T}}^{'})= \mathcal{S}_{k}$, for some $ k \in \{1,\cdots,\beta(\mathcal{S})\}.$
\end{thm}

We recall that the definition of a flat $S$-act is that it is $\mathcal{C}$-flat where
$\mathcal{C}$ is the class of {\em all} embeddings of right $S$-acts.

By Lemma~\ref{lem:allfree}, the class of all right $S$-acts has Condition (Free),  so from
Theorem~\ref{thm:FoutP}, we immediately have the following corollary. Note the extra equivalent condition, to bring it into line with \cite[Theorem 12]{bulmanfleminggould}.

\begin{cor}\label{cor:flat} \cite{bulmanfleminggould} The following conditions are equivalent for a monoid $S$:

 (1) the class $\mathcal{F}$ is axiomatisable;
 
 (2) the class $\mathcal{F}$ is closed under formation of ultraproducts;
 
  (3) for every skeleton $\mathcal{S} \in \mathbb{S}$ there exist finitely many replacement skeletons $\mathcal{S}_{1},\cdots,\mathcal{S}_{\alpha(\mathcal{S})}$ such that, for any right $S$-act 
  embedding $\gamma:A\rightarrow A'$, and any flat left $S$-act $B$, if $(a\gamma  ,b),\, (a'\gamma , b') \in A' \times B$  are connected by a tossing  $\mathcal{T}$ over $A'$ and $B$ with $ \mathcal{S}(\mathcal{T})= \mathcal{S}$, then $(a,b)$ and $(a',b')$ are connected by a tossing ${\mathcal{T}}^{'}$ over $A$ and $B$ such that $\mathcal{S}({\mathcal{T}}^{'})= \mathcal{S}_{k}$, for some $ k \in \{1,\cdots,\alpha(\mathcal{S})\}$;
 
 (4) for every skeleton $\mathcal{S} \in \mathbb{S}$ there exist finitely many replacement skeletons $\mathcal{S}_{1},\cdots,\mathcal{S}_{\alpha(\mathcal{S})}$ such that, for any right $S$-act $A$ and any flat left $S$-act $B$, if $(a  ,b),\, (a' , b') \in A \times B$  are connected by a tossing  $\mathcal{T}$ over $A$ and $B$ with $ \mathcal{S}(\mathcal{T})= \mathcal{S}$, then $(a,b)$ and $(a',b')$ are connected by a tossing ${\mathcal{T}}^{'}$ over $aS \cup a'S$ and $B$ such that $\mathcal{S}({\mathcal{T}}^{'})= \mathcal{S}_{k}$, for some $ k \in \{1,\cdots,\alpha(\mathcal{S})\}$;
 
 (5) for every skeleton $\mathcal{S} \in \mathbb{S}$ there exists finitely many replacement skeletons $\mathcal{S}_{1},\cdots,\mathcal{S}_{\beta(\mathcal{S})}$ such that, for any flat left $S$-act $B$, if $([x],b)$ and $([x'],b') $ are connected by a tossing $\mathcal{T}$ over $F^{m+1}/\rho_{\mathcal{S}}$ and $B$ with $\mathcal{S}(\mathcal{T}) = \mathcal{S}$, then $([x], b)$, and $([x'],b')$ are connected by a tossing ${\mathcal{T}}^{'}$ over $[x]S \cup [x']S$ and $B$ such that $\mathcal{S}({\mathcal{T}}^{'})= \mathcal{S}_{k}$, for some $ k \in \{1,\cdots,\beta(\mathcal{S})\}.$

\end{cor}

\subsection{Axiomatisability of $\mathcal{CF}$ in general case}\label{subsec:sactsoutFP}

We continue to consider a class $\mathcal{C}$ of embeddings of right $S$-acts, but now drop our assumptions that Condition (Free)  holds. The results and proofs of this section are analogous to those for weakly flat  $S$-acts in \cite{bulmanfleminggould}. Note that the conditions in (3) below appear weaker than those in 
Theorem~\ref{thm:FoutP}, as we are only asking that for specific elements $a,a'$ and skeleton
$\mathcal{S}$, there are finitely many replacement skeletons, in the sense made specific below.

\begin{thm}\label{thm:sactsoutFP} Let $\mathcal{C}$ be a class of embeddings of right $S$-acts.

The following conditions are equivalent:

(1) the class $\mathcal{CF}$ is axiomatisable;

(2) the class $\mathcal{CF}$ is closed under ultraproducts;

(3) for every skeleton $\mathcal{S}$ over $S$ and $a,a' \in A$, where $\mu:A \to A'$ is in $\mathcal{C}$, there exist finitely many skeletons $\mathcal{S}_{1},\cdots,\mathcal{S}_{\alpha(a,\mathcal{S},a',\mu)}$, such that for any $\mathcal{C}$-flat left $S$-act $B$, if $(a \mu , b),\,(a' \mu ,b')$ are connected by a tossing $\mathcal{T}$ over $A'$ and $B$ with $\mathcal{S}(\mathcal{T})= \mathcal{S}$, then $(a,b)$ and $(a',b')$ are connected by a tossing $\mathcal{T}'$ over $A$ and $B$ such that $\mathcal{S}(\mathcal{T}')=\mathcal{S}_{k}$, for some $ k \in \{1,\cdots,\alpha(a , \mathcal{S},a',\mu)\}$.
\end{thm}

We now explain why the axiomatisability of weakly flat $S$-acts as given in \cite{bulmanfleminggould} then becomes a special case.
We recall that a left $S$-act $B$ is called {\em weakly flat}
 if the functor $- \otimes B $ maps  embeddings
of right ideals in the category of ${\mathbf S}${\bf-Act} to
 one-one maps in the category of  {\bf Set}. So, $B$ is weakly flat
 if it is $\mathcal{C}$-flat where $\mathcal{C}$ is the class of all embeddings of
 right ideals of $S$ into $S$.  In our Corollary, we do not need to mention the embeddings
 $\mu$, since they are all inclusion maps of right ideals into $S$.

\begin{cor}\label{cor:axpweaklyflat}\cite[Theorem 13]{bulmanfleminggould}
The following are equivalent for a monoid $S$:

(i) the class $\mathcal{WF}$ is axiomatisable;

(ii) the class $ \mathcal{WF}$ is closed under ultraproducts;

(iii) for every skeleton $\mathcal{S}$ over $S$ and $a,a' \in S$ there exists finitely many skeletons $\mathcal{S}_1,\cdots, \mathcal {S}_{\beta(a,\mathcal{S},a')} $ over $S$, such that for any weakly flat left $S$-act $B$, if $(a,b)$, $(a',b') \in S \times B$ are connected by a tossing $\mathcal{ T}$ over $S$ and $B$ with $ \mathcal{S(T)}= \mathcal{S}$ then $(a,b)$ and $(a',b')$ are connected by a tossing $\mathcal{T'}$ over $aS \cup a'S$ and $B$ such that $ \mathcal{S(T')}=\mathcal{S}_k$ for some $k \in \{1,\cdots,\beta(a,\mathcal{S},a' )\}$.
\end{cor}

 We say that  a left $S$-act $B$ is
{\em principally weakly flat} if it is $\mathcal{C}$-flat where $\mathcal{C}$ is the set
of all embeddings of principal right ideals of $S$ into $S$. We end this section by considering the
axiomatisability of principally weakly flat $S$-acts. We first remark that if $aS$ is a principal
right ideal of $S$ and $B$ is a left $S$-act, then
\[au\otimes b=av\otimes b'\mbox{ in }aS\cup B\mbox{ if and only if }
a\otimes ub=a\otimes vb'\mbox{ in }aS\otimes B\]
with a similar statement for $S\otimes B$. Thus $B$ is principally weakly flat if and only if for
all $a\in S$, if $a\otimes b=a\otimes b'$ in $S\otimes B$, then
$a\otimes b=a\otimes b'$ in $aS\otimes B$.

Our next result follows from Theorem~\ref{thm:sactsoutFP} and its proof.

\begin{cor}\label{cor:axiopwfsact}
The following conditions are equivalent for a monoid $\mathcal{S}$:

(i) the class $ \mathcal{PWF}$ is axiomatisable;

(ii) the class $\mathcal{PWF}$ is closed under ultraproducts;

(iii) for every skeleton $\mathcal{S}$ over $S$ and $ a \in S$ there exists finitely many skeletons $ \mathcal{S}_{1},\cdots, \mathcal{S}_{\tau( a, \mathcal{S})}$ over $S$, such that for any principally weakly flat left $S$-act $B$, if $(a,b)$, \, $(a,b') \in S \otimes B$ are connected by a tossing $\mathcal{T}$ over $S$ and $B$ with $\mathcal{S(T)=S}$, then $(a,b)$ and $(a,b')$ are connected by a tossing $\mathcal{T}'$ over $aS $ and $B$ such that $\mathcal{S(\mathcal{T}')}= \mathcal{S}_{k}$ for some $k \in \{ 1,\cdots,\tau(a, \mathcal{S}) \}$.
\end{cor}

\section{Axiomatisability of specific classes of $S$-acts}\label{sec:axsacts}

We now examine specific classes of $S$-acts 
 which can be axiomatisable by various techniques. Axiomatisability of classes $\mathcal{E}$ and $\mathcal{P}$ using the ``elements'' method are given in \cite{gould}, we will be discussing axiomatisability of these classes by using ``replacement tossings'' methods here.

 To axiomatise classes such as $\mathcal{EP},\,\mathcal{W},\,\mathcal{PWP}$  we use  
 both methods of proof, i.e.  ``elements'' and  ``replacement tossings'' methods.

\subsection{Axiomatisability of Condition $(P)$ for $S$-acts}\label{sec:axiocondpsact}

Let $S$ be a monoid. For any $s,t\in S$ we put
\[\mathbf{R}(s,t)=\{ (u,v)\in S\times S:su=tv\}\]
and notice that $\mathbf{R}=\emptyset$ or is an $S$-subact of $S\times S$.

 The following result is implicit in \cite{gould} and made explicit in \cite{gouldpalyutin}.

\begin{thm}\label{thm:P}\cite{gould,gouldpalyutin} The following conditions are equivalent for a monoid $S$:

(i) the class $\mathcal{P}$ is axiomatisable;

(ii)  the class $\mathcal{P}$ is closed under ultraproducts;

(iii)  the class $\mathcal{P}$ is  closed under ultrapowers;

(iv) every ultrapower of $S$ lies in $\mathcal{P}$;

(v) for any $s,t\in S$, $\mathbf{R}(s,t)=\emptyset$ or is finitely generated. 
\end{thm}

We now rephrase the above in terms of replacement tossings.

\begin{rem}\label{rem:length2} Observe that if $sa = tb$ for some $s, t \in S, a, b \in B$, then

\[
\begin{array}{rclcrcl}
                  &    &                       & a& = & 1 \, a\\
                 s \, 1 & = & 1 \,s       & s \, a & = & t \, b \\  
                1 \, t & =& t \,1     & 1 \, b &= & b \\
                                                                                                   
\end{array}\]
so $(s,a),\,(t,b)$ are connected via a tossing of length $2$ over $S$ and $B$ with skeleton $(1,\,s,\,t,\,1)$.

Conversely if $(s,a),\, (t,b)$ are connected with skeleton $(1,\,s,\,t,\,1)$ in the way

\[
\begin{array}{rclcrcl}
                  &    &                       & a& = & 1 \, b_1\\
                 s \, 1 & = & a_2 s        & s \, b_1 & = & t \, b_2 \\  
                a_2\,t & =& t \,1     & 1 \, b_2 &= & b \\
                                                                                                   
\end{array}\]
then $sa = tb.$

\end{rem}
\begin{rem}\label{rem:length1}

Suppose  $su = tv ,\,a= uc , b = vc $ then 

\[
\begin{array}{rclcrcl}
                  &    &                       & a& = & u \, c\\
                 su &= & tv       & vc & = & b \\                                                                                                    
\end{array}\]
is a length $1$ tossing connecting $(s,a ) $ to $ (t,b)$ over $S$ and $B$ or over $(sS \cup tS)$ and $B$ with skeleton $(u,v)$.

Conversely if  there exists length $1$ tossing connecting $(s,a)$ to $(t,b)$ over $S$ and $B$ 
with skeleton $(u,v)$ it must look like 

\[
\begin{array}{rclcrcl}
                  &    &                       & a& = & u \, b_1\\
                 su &= &tv     & v b_1 & = & b. \\                                                                                                    
\end{array}\]

so $(u,v) \in \mathbf{R}(s,t)$.
\end{rem}

\begin{cor}\label{cor:axiocondpsacts}

The following conditions are equivalent for a monoid $S$:

(i) the class $\mathcal{P}$ is axiomatisable;

(ii) for every skeleton $\mathcal{S}=(1,\,s,\,t,\,1)$  over $S$, there exists finitely many replacement skeletons $\mathcal{S}_{1}=(u_1,v_1), \cdots, \mathcal{S}_{n({\mathcal{S}})}=(u_{n(\mathcal{S})},
v_{n(\mathcal{S})})$ of length one such that for any $a,b \in B \in \mathcal{P}$ and $sa = tb$ (equivalently,  $(s,a)$ is connected to $ (t,b)$ via a tossing with skeleton $\mathcal{S}$), then $(s,a)$ is connected to $(t,b)$ via a replacement tossing  
with skeleton $\mathcal{S}_{i}$, for some $1 \leq i \leq n(\mathcal{S}).$ 

\end{cor}
\begin{proof} Suppose that $(i)$ holds. Let $\mathcal{S}=(1,s,t,1)$ be a skeleton. From 
Theorem~\ref{thm:P}, $\mathbf{R}(s,t)=\emptyset$ or
$\mathbf{R}(s,t)$ is finitely generated. In the first case, set
$n(\mathcal{S})=0$ and in the second, suppose that
\[\mathbf{R}(s,t)=\bigcup_{i=1}^{i=n}(u_i,v_i)S.\]
Put $n(\mathcal{S})=n$ and let $\mathcal{S}_i=(u_i,v_i)$ for $1\leq i\leq n$.

Let $B\in \mathcal{P}$ and suppose that $sa=tb$ for some $s,t\in S$ and $a,b\in B$.
Then $ss'=tt'$ and $a=s'c$, $b=t'c$ for some $s',t'\in S$ and $c\in B$. But then
$(s',t')=(u_i,v_i)r$ for some $i\in \{ 1,\hdots ,n\}$ and $r\in S$, so that
$a=u_id, b=v_id$ for some $d=rc\in B$ and $(u_i,v_i)$ is the skeleton of a replacement tossing. Hence $(ii)$
holds.

Conversely, suppose that $(ii)$ holds. If $\mathbf{R}(s,t)\neq\emptyset$, let
$(u,v)\in \mathbf{R}(s,t)$. Then $su=tv$ and as $S\in \mathcal{P}$ we have that
there is a replacement tossing with skeleton $(u_i,v_i)$ connecting $(s,u)$ to $(t,v)$. Perforce
we have that $(u_{i},v_{i}) \in \mathbf{R}(s,t), \,u=u_ic,v=v_ic$ so that $(u,v)=(u_i,v_i)c$ for some $c\in S$. It follows
that $\mathbf{R}(s,t)$ is finitely generated. By Theorem~\ref{thm:P}, $\mathcal{P}$ is axiomatisable.
\end{proof}

\subsection{Axiomatisability of Condition $(E)$ for $S$-acts}{\label{sec:axiocondEsact}}

Let $S$ be a monoid. For any $s,t\in S$ we put
\[\mathbf{r}(s,t)=\{ u \in S :su=tu\}\]
and notice that $\mathbf{r}(s,t)=\emptyset$ or is a right ideal of $S$. 

 The following result is given in \cite{gould,gouldpalyutin}.

\begin{thm}\label{thm:E}\cite{gould,gouldpalyutin} The following conditions are equivalent for a monoid $S$:

(i) the class $\mathcal{E}$ is axiomatisable;

(ii)  the class $\mathcal{E}$ is closed under ultraproducts;

(iii)  the class $\mathcal{E}$ is  closed under ultrapowers;

(iv) every ultrapower of $S$ lies in $\mathcal{E}$;

(v) for any $s,t\in S$, $\mathbf{r}(s,t)=\emptyset$ or is finitely generated. 
\end{thm}

We now rephrase the above in terms of replacement tossings.

\begin{rem}\label{rem:length1e} Suppose $su = tu$, $a = uc$ then 

\[
\begin{array}{rclcrcl}
                  &    &                       & a& = & u\, c\\
                 su &= &tu       & u c & = & a \\                                                                                                    
\end{array}\]
is a length $1$ tossing connecting $(s,a) $ to $ (t,a)$ over $S$ and $B$ with skeleton $(u,u)$. We will say a skeleton of the form  $(u,u)$ a {\em trivial} skeleton.

Conversely if  there exists length $1$ tossing connecting $(s,a) $ to $(t,a)$ over $S$ and $B$ by a trivial skeleton $(s_1,s_1)$ it must look like 

\[
\begin{array}{rclcrcl}
                  &    &                       & a& = & s_1 \, b_1\\
                 ss_1 &= &ts_1       & s_1 b_1 & = & a \\                                                                                                    
\end{array}\]

notice that $s_1 \in \mathbf{r}(s,t)$.
\end{rem}

\begin{cor}\label{thm:axiocondEsact}
The following conditions are equivalent for a monoid $S$:

(i) the class $\mathcal{E}$ is axiomatisable;

(iii) for every skeleton $\mathcal{S}=(1,\,s,\,t,\,1)$  over $S$, there exists finitely many trivial replacement skeletons $\mathcal{S}_{1}=(u_1,u_1), \cdots, \mathcal{S}_{m(\mathcal{S})}=(u_{m(\mathcal{S})},u_{m(\mathcal{S})})$  such that for any $a \in B \in \mathcal{E}$ and $sa = ta$ (equivalently,  $(s,a)$ is connected to $(t,a)$ via a tossing with skeleton $\mathcal{S}$), then $(s,a)$ is connected to $(t,a)$ via a replacement tossing  of trivial skeleton  $\mathcal{S}_{i}$, $1 \leq i \leq m(\mathcal{S}).$

\end{cor}
\begin{proof} Suppose that $(i)$ holds. Let $\mathcal{S}=(1,s,t,1)$ be a skeleton. From 
Theorem~\ref{thm:E}, $\mathbf{r}(s,t)=\emptyset$ or
$\mathbf{r}(s,t)$ is finitely generated right ideal of $S$. In the first case, set
$m(\mathcal{S})=0$ and in the second, suppose that
\[\mathbf{r}(s,t)=\bigcup_{i=1}^{i=m}(u_i,v_i)S.\]
Put $m(\mathcal{S})=m$ and let $\mathcal{S}_i=(u_i,v_i)$ for $1\leq i\leq m$.

Let $A \in \mathcal{E}$ and suppose that $sa=ta$ for some $s,t\in S$ and $a\in A$.
Then $ss'=ts'$ and $a=s'c$ for some $s'\in S$ and $c\in A$. But then
$(s',s')=(u_i,v_i)r$ for some $i\in \{ 1,\hdots ,m\}$ and $r\in S$, so that
$a=u_id$ for some $d=rc\in A$ and $(u_i,u_i)$ is the trivial skeleton of a replacement tossing. Hence $(ii)$
holds.

Conversely, suppose that $(ii)$ holds. If $\mathbf{r}(s,t)\neq\emptyset$, let
$u \in \mathbf{r}(s,t)$. Then $su=tu$ and as $S\in \mathcal{E}$ we have that
there is a replacement tossing with trivial skeleton $(u_i,u_i)$ connecting $(s,u)$ to $(t,u)$. 
We have that $u_{i} \in \mathbf{r}(s,t)$ and $u=u_ic$ which gives
that $\mathbf{r}(s,t)$ is finitely generated. By Theorem~\ref{thm:E}, $\mathcal{E}$ is axiomatisable.
\end{proof}

\subsection{Axiomatisability of Condition $(EP)$}\label{subsec:ep}
In \cite{golchin} Akbar Golchin and Hossein Mohammadzadeh defined a new flatness property of acts over monoids which is an extended version of Conditions (E) and (P). Moreover they have shown the following relations exist among Conditions  $(E),\,(P)$ and $(EP)$.

\begin{rem}\cite{golchin} Condition $(E) $ implies  Condition $ (EP)$ and Condition  $(P)$ implies Condition $ (EP)$ but neither converse is true. 
\end{rem}

\begin{thm}\label{thm:(EP)} 
The following conditions are equivalent for a  monoid $S$:

(i) the class $\mathcal{EP}$ is axiomatisable;

(ii) $\mathcal{EP}$ is closed under ultraproducts;

(iii) $\mathcal{EP}$ is closed under ultrapowers;

(iv)  for any $s,t \in S$ either $sa \neq ta$ for all $ a \in A 
\in \mathcal{EP}$ or there exists $f \subseteq \mathbf{R}(s,t)$,\, $ f $ is finite, such that if
$$ sa = ta,\, a \in A \in \mathcal{EP} \,\,\mbox{then} \,\, (a,a)=(u,v)t$$ for some $(u,v) \in f$ and $t \in A.$

\end{thm}

\begin{proof}

The proof is along the similar lines as given for ordered case in \cite{gouldshaheen1}, for class $\mathcal{EP}^{\leq}$, or see detail proof in \cite{shaheen:2010}.

\end{proof}

\begin{rem}\label{rem:EP}

Note that for any $ a \in A \in {\mathcal{EP}}$, if $sa =sa$ then certainly $(a, a) = (1, 1)a$ and
$(1, 1) \in \mathbf{R}(s, s).$ So that to check the condition $(iv)$ of Theorem~\ref{thm:(EP)} holds, it is
enough to consider the cases where $s \not = t.$

 If S is a monoid such that $\mathbf{R}(s, t)$ is finitely generated for all $s, t \in S$
with $s \not = t$, then $\mathcal{EP}$ is axiomatisable. To see this let $S$ be a monoid such that $\mathbf{R}(s,t)$ is finitely generated, let $sa= ta$ for some $a \in A \in \mathcal{EP}$, then $a= u a'=va'$ for some $u,v \in S$ and $a' \in A$ with $su= tv$, so that $(u,v)=(u_{i},v_{i})t$ for some $t \in S$ and $ i \in \{1,\cdots,n\}.$  Now  $a= u_i  a''= v_i  a''$ where we can choose $f=\{(u_1,v_1),\cdots,(u_n,v_n)\}$, condition $(iv)$ of Theorem ~\ref{thm:(EP)} satisfied.

We can conclude that if $\mathcal {P}$ is axiomatisable, so is $\mathcal{EP}$.

\end{rem}

\begin{rem}\label{rem:length1ep} Suppose  $su = tv$, $a = ua''= va''$ then 

\[
\begin{array}{rclcrcl}
                  &    &                       & a& = & u\, a''\\
                 su &= &tv       & v a'' & = & a \\                                                                                                    
\end{array}\]
is a length $1$ tossing connecting $(s,a) $ to $ (t,a)$ over $S$ and $B$ with skeleton $(u,v)$.

Conversely if there exists length $1$ tossing connecting $(s,a) $ to $ (t,a)$ over $S$ and $B$ by a  skeleton $\mathcal{S}=(s_1,t_1)$ it must look like 

\[
\begin{array}{rclcrcl}
                  &    &                       & a& = & s_1 \, a_1\\
                 ss_1 &= &tt_1       & t_1 a_1 & = & a \\                                                                                                    
\end{array}\]

so that $(s_1,t_1) \in \mathbf{R}(s,t)$.
\end{rem}

\begin{rem} From Remark ~\ref{rem:length2} it is obvious that $sa=ta$ if and only if $(s,a)$ connected to $(t,a)$ over $S$ and $B$ via a tossing of length $2$ with skeleton $(1,s,t,1)$.
\end{rem}

\begin {cor}\label{thm:axiocondEPsact}
The following conditions are equivalent for a monoid $S$:

(i) the class $\mathcal{EP}$ is axiomatisable;

(ii) for every skeleton $\mathcal{S}=(1,\,s,\,t,\,1)$  over $S$, there exists finitely many  replacement skeletons $\mathcal{S}_{1}=(u_1,v_1), \cdots, \mathcal{S}_{p(\mathcal{S})}=(u_{p(\mathcal{S})},v_{p(\mathcal{S})})$  such that for any $a \in B \in \mathcal{EP}$ and $sa = ta$ (equivalently,  $(s,a)$ is connected to $(t,a)$ via a tossing with skeleton $\mathcal{S}$), then $(s,a)$ is connected to $(t,a)$ via a replacement tossing  of  skeleton  $\mathcal{S}_{i}$, $1 \leq i \leq p(\mathcal{S}).$

\end{cor}
\begin{proof} We follow the similar arguments as given in the proof of Corallaries ~\ref{cor:axiocondpsacts} and ~\ref{thm:axiocondEsact}.
\end{proof}

\subsection{Axiomatisability of Condition $(W)$ for $S$-acts}\label{sec:condwsact}
 In \cite{kilpknauer} Bulman-Fleming and McDowell introduced an interpolation type condition called Condition (W). We will describe the condition on a monoid $S$ such that $\mathcal{W}$ is axiomatisable.  
 
 We remind reader the following definition:
\begin{defi} \label{defi:w}

A left $S$-act $A$ satisfies {\em Condition (W)}, if whenever $sa = t a'$ for $ a,a' \in A$, $s,t \in S$, then there exists $ a'' \in A$, $ u \in sS \cap tS$, such that $sa= ta'=u a''$. We will denote the class of left $S$-acts satisfying Condition (W) as $\mathcal{W}$.
\end{defi}

\begin{rem}\label{rem:Sw} The monoid $S$ satisfies Condition (W) as an $S$-act.
\end{rem}

\begin{thm}\label{thm:condwsact}
The following conditions are equivalent for a monoid $S$:

(i) the class $\mathcal{W}$ is axiomatisable;

(ii) $\mathcal{W}$ is closed under ultraproducts;

(iii) $\mathcal{W}$ is closed under ultrapowers;

(iv) every ultrapower of $S$ lies in $\mathcal{W}$;

(v) for any $ s, t \in S$, $ s S \cap t S =  \emptyset $ or $ s S \cap t S$ is finitely generated as a right ideal of $S$.

\end{thm}
\begin{proof}

$(i)$ implies $(ii)$: this follows from \L os's Theorem; 
$(ii)$ implies $(iii) $ is clear and $(iii)$ implies $(iv)$ is obvious as $S$ satisfies Condition $(W)$ as an $S$-act by Remark~\ref{rem:Sw}.

$(iv)$ implies $(v)$: let $ s , t \in S$ and suppose that $ sS \cap tS \not = \emptyset $. Clearly $ sS \cap tS$ is a right ideal of  $S$ and so in particular is a right $S$-act. We suppose that $ sS \cap tS$ is not finitely generated. Let $ \{u_{\beta}: \beta < \gamma \} $ be a generating subset of $ sS \cap  tS$ of cardinality $\gamma$, where $u_{\beta}= s x_{\beta} = t y_{\beta} $ for some $x_{\beta}$,\,$y_{\beta}$ in $S$.

By assumption $\gamma $ is a limit ordinal. We may suppose that for any $ \beta < \gamma $ , $u_{\beta} $ is not in the right ideal generated by the preceding elements $u_\tau$ that is $u_\beta \not \in \bigcup_{\tau < \beta } u_{\tau }S$, for any $\beta < \gamma $.

Let $\Phi$ be a uniform ultrafilter on $\gamma $, that is $\Phi$ is an ultrafilter on $\gamma$ such that all sets in $\Phi$ have same cardinality $\gamma$. Let ${\mathcal U} = S ^{\gamma} / \Phi$. By assumption, ${\mathcal U}$ satisfies Condition (W) as a $S$-act.

Define elements $\underline{a}=(x_{\beta})$ and $\underline{b}=(y_{\beta})$ and consider ${\underline{a}}_{\Phi},\,{\underline{b}}_{\Phi} \in \mathcal{U}$. Since $s x_{\beta}=u_\beta =t y_\beta $ for all $ \beta < \gamma$, clearly $ s {\underline{a}}_{\Phi} = t {\underline{b}}_{\Phi}$. By assumption $ {\mathcal U}$ satisfies Condition (W) so there exists $ {\underline{c}}_{\Phi} \in {\mathcal U} $ and $ u  \in sS \cap tS$ such that $s {\underline{a}}_{\Phi} = t {\underline{b}}_{\Phi} = u\,{\underline{c}}_{\Phi}$. Let ${\underline{c}}_{\Phi}=(z_{\beta})_{\Phi}$ so there exists sets $T_1, T_2 \in \Phi$ such that $s \,x_{\beta} = u\,z_{\beta} $  for all $ \beta \in T_1$ and $t \,y_{\beta}=  u \,z_{\beta}$ for all $  \beta \in T_2$.  Since $u \in sS \cap tS$ there exists $\sigma < \gamma $ and $ h \in S$ with $ u = u_\sigma \, h$.
Using the fact that $ T_1 \cap T_2 \in \Phi$ and $\Phi$ is uniform ultrafilter, $T_1 \cap T_2$ contains an ordinal say $ \alpha \geq \sigma +1 $. Then
$$u_\alpha = s \, x_{\alpha} = t \, y_{\alpha} = u\, z_\alpha =u_{\sigma} \,h \, z_{\alpha}$$ and so \, $u_\alpha \in u_{\sigma} S$, a contradiction.  
Thus $sS \cap tS$ is finitely generated.

$(v)$ implies $(i)$: we show that the class of $S$-acts satisfying Condition (W) is axiomatisable by giving explicitly a set of sentences that axiomatises $\mathcal{W}$.

For any element $\rho=(s,t) $ of $S \times S$ with $ sS \cap tS \not = \emptyset $, we choose and fix a finite set of generators $\{u _{\rho, 1}, \cdots, u_{\rho, n(\rho)}\}$ of $sS \cap tS$. For $s,t \in S$ we define sentences $\Upsilon_{\rho}$ , as follows:

If $sS \cap tS = \emptyset$ then $ \Upsilon_{\rho}$ is 
$$( \forall x)(\forall y)( sx \not = ty);$$
if $sS \cap tS \not = \emptyset $ then $\Upsilon_{\rho}$ is 

$$(\forall x)(\forall y)\bigg(sx = ty \rightarrow (\exists z)( \bigvee_{i=1}^{n(\rho)} sx =ty= u_{\rho,i} \,z )\bigg)$$
Let 
$$\Sigma_{\mathcal{W}} = \Sigma_S \cup \,\{ \Upsilon_{\rho}: \rho \in S \times S \}.$$

\bigskip

\noindent We claim that $ \Sigma_{\mathcal{W}}$ axiomatises  $\mathcal{W}$.

\vspace{2mm}


Suppose that $A$ is a $S$-act satisfying Condition (W) and $\rho \in S \times S$, where $\rho = (s,t)$. If $sS \cap tS = \emptyset $ and there exists $a,b \in A$ such that $sa =tb$, then since $A$ satisfies Condition (W), there exists $u \in sS \cap tS$ (such that $sa = tb = u c $  for some $ c \in A$), a contradiction. Thus $A \models \Upsilon_{\rho}$. 

If $ sS \cap tS \not = \emptyset$ and $ sa = tb $ where $ a, b \in A$ then again using the fact that $A$ satisfies Condition (W) there are elements $ u \in sS \cap tS$ and $ a' \in A$ such that $sa = tb = u a'$. Now $u \in sS \cap tS$ and so $ u = u_{\rho, i} h $ for some $ i \in \{ 1,2, \cdots, n (\rho) \} $ and $ h \in S$. Thus $sa= tb = u_{\rho, i} h a'$, where $ h a' \in A$. Hence $A \models \Upsilon_{\rho} $.

Conversely if $A$ is a model of $\Sigma_{\mathcal{W}}$ and if $sa = tb$ where $ s,t \in S$ and $ a,b \in A$, then since $ A \models \Upsilon_{\rho}$, where $\rho=(s,t)$ it follows that $ sS \cap tS$ cannot be empty and $\Upsilon_{\rho}$ is 

$$(\forall x)(\forall y)\big(sx = ty \rightarrow (\exists z)( \bigvee_{i=1}^{n(\rho)} sx =ty= u_{\rho,\, i} z )\big).$$ 
Hence there exists an element $ c \in A$ such that $ sa = tb = u_{\rho, i} c $ for some $i \in \{ 1,2, \cdots, n(\rho) \}$. By definition of $ u_{\rho,\, i }  $ we have $ u_{\rho,\, i } \in sS \cap tS $. Thus $A$ satisfies Condition (W)  and so $\Sigma_{\mathcal{W}}$ axiomatises $\mathcal{W}$.

\end{proof}

We now explain the axiomatisability of Condition (W) in terms of replacement skeletons.

\begin{rem}\label{rem:length2w} Observe that if $sa = tb=ua'$ for some $s, t,u \in S, a, b,a' \in B$, then

\[
\begin{array}{rclcrcl}
                  &    &                       & a& = & 1 \, a\\
                 s \, 1 & = & 1 \,s       & s  a & = & u a' \\  
                   1\, u     & =  & 1\, u                & ua' &= & tb\\
                   1\, t&=&t\, 1&1\, b&=&b
                                                                                                   
\end{array}\]
so $(s,a),\,(t,b)$ are connected via a tossing of length $3$ over $S$ and $B$ with skeleton
$(1,s,u,u,t,1)$. 

Conversely if  $(s,a)$ and $(t,b)$ are via a tossing with skeleton
$(1,s,u,u,t,1)$, we have
\[
\begin{array}{rclcrcl}
                  &    &                       & a& = & 1 \, b_1\\
                 s \, 1 & = & a_2 s        & s  b_1 & = & u b_2 \\  
                a_2u & =& a_3u    & u b_2 &= & tb_3 \\
                a_3t&=&t\, 1&1\, b_3&=&b,   
\end{array}\]then $sa=tb=ub_2$ for some $b_2 \in B$.

\end{rem}

\begin{cor}\label{cor:axiocondwsacttossing}
The following conditions are equivalent for a monoid $S$:

(i) the class $\mathcal{W}$ is axiomatisable;

(ii) for every skeleton $\mathcal{S}=(1,\,s,\,t,\,1)$  over $S$, there exists finitely many  replacement skeletons $$\mathcal{S}_1=(1,s,u_1,u_1,t,1), \hdots ,\mathcal{S}_{n(\mathcal{S})}=(1,s,u_{n(\mathcal{S})},u_{n(\mathcal{S})},t,1)$$
where $u_i\in sS\cap tS$,
such that for any $a,b \in B \in \mathcal{W}$ and $sa = tb$ (equivalently,  $(s,a)$ connected with $ (t,b)$ via a tossing over $S$ and $B$ with skeleton $\mathcal{S}$), then $(s,a)$ is connected to $(t,b)$ connected 
via a tossing over $S$ and $B$ with skeleton $\mathcal{S}_i$(equivalently,
$sa=tb=u_id$ for some $d\in B$), for some $1\leq i\leq n$. 
\end{cor}
\begin{proof} Suppose that $\mathcal{W}$ is axiomatisable. Let
$\mathcal{S}=(1,s,t,1)$ be a skeleton. If $sS\cap tS=\emptyset$, we put $n=0$.
Otherwise, we know from Theorem~\ref{thm:condwsact} that $sS\cap tS$ is finitely generated, say
by $u_1,\hdots ,u_n$. Let $\mathcal{S}_i=(1,s,u_i,u_i,t,1)$. If $B\in \mathcal{W}$ and
$sa=tb$, then $sa=tb=vc$ for some $v\in sS\cap tS$ and $c\in B$. But then $v=u_ir$ for some
$i\in \{ 1,\hdots ,n\}$, giving $sa=tb=u_id=u_i rc$. Thus $(ii)$ holds by Remark
~\ref{rem:length2w}.

Conversely, suppose that $(ii)$ holds and let $s,t\in S$. Suppose that $sS\cap tS\neq \emptyset$ and
let $r=sa=tb$ where $r,a,b\in S$. Certainly $S$ has Condition (W), so that there is a replacement tossing
$\mathcal{S}_i=(1,s,u_i,u_i,t,1)$ for some $u_i\in sS\cap tS$ and $i\in \{ 1,\hdots ,n\}$. Hence
$r=sa=tb=u_id$ for some $d\in S$ so that $r\in u_iS$ and we deduce $sS\cap tS=\bigcup_{1\leq i\leq n}u_iS$.
By Theorem~\ref{thm:condwsact}, $\mathcal{W}$ is axiomatisable.
\end{proof}
 \begin{rem} We have replaced a smaller tossing with longer one. This is concerned with having a common tossing $(s,a) \to (u_i,d)$ and $(t,b) \to (u_i,d)$.
 \end{rem}

\subsection{Axiomatisability of Condition (PWP)}

 We  note that $\mathbf{R}(t,t)$ is as follows: 
 
 $$\mathbf{R}(t,t) = \{ (u,v ) \in S \times S: tu = tv\}.$$

 \begin{rem} Note that $S$ satisfies Condition $(PWP)$.
 \end{rem}
 

 \begin{thm}\label{thm:PWP} The following conditions are equivalent for a monoid $S$:

 (i) the class $\mathcal{PWP}$ is axiomatisable;

(ii) $\mathcal{PWP}$ is closed under ultraproducts;

(iii) $\mathcal{PWP}$ is closed under ultrapowers;

(iv) every ultrapower of $S$ has $\mathcal{PWP}$;
 
 (v) for any $t \in S,\,\, \mathbf{R}(t,t) = \emptyset$ or $\mathbf{R}(t,t)$ is  finitely generated as an $S$-subact of $S \times S$.
 
 \end{thm}
 \begin{proof}

 $(i)$ implies $(ii)$: this follows from \L os's Theorem, $(ii)$ implies $(iii) $ is clear;
$(iii)$ implies $(iv)$ is obvious as $S$ satisfies Condition $(PWP)$ as an $S$-act. 
.

$(iv)$ implies $(v)$: suppose $\mathbf{R}(t,t) \not = \emptyset$ and is not finitely generated. Suppose for  each finite subset $f $ of $ \mathbf{R}(t,t)$, there exists $a_{f},a^{'}_{f} \in S$ with  $t a_{f}= t a^{'}_{f}$ and $(a_{f},a^{'}_{f}) \not \in \,f \,S.$


Let $J$ be the set of finite subsets of $\mathbf{R}(t,t)$.  For each $(u,v) \in \mathbf{R}(t,t)$ we define $$J_{(u,v)}= \{f \in J:(u,v) \in f\}.$$

As each intersection of finitely many of the sets $J_{(u,v)}$ is non-empty, so we are able to define an ultrafilter $\Phi$ on $J$,  such that each $J_{(u,v)} \in \Phi$ for all $(u,v) \in \mathbf{R}(t,t).$

Let $\underline{a}=(a_f)$ and $\underline{a}^{'}=(a'_f)$ then  $t \underline{a}= t \underline{a}^{'}$ in $\mathcal{S}$ where $\mathcal{S}= \prod_{f \in J}S^{f}$, where each $S^{f}$ is a copy of $S$, and this equality is determined by a product  of the elements $t a_{f}= t a^{'}_{f}$ with $a_{f}, a^{'}_{f}\in S$, for each $f \in J$. It follows that this equality  $t \underline{a}_{\Phi}= t \underline{a}^{'}_{\Phi}$ also holds in $\mathcal{U}$ where  $\mathcal{U} = \prod_{f \in J} S^{f} /\Phi$; by assumption  $\mathcal{U}$ has $\mathcal{PWP}$,  so there exists $u , v \in S$, and ${\underline{r}}_{\Phi}=(r_{f})_{\Phi} \in \mathcal{U}$ such that  $$\underline{a}_{\Phi} = u \underline{r}_{\Phi}, \,\,\,a^{'}_{\Phi}= v \underline{r}_{\Phi},\,\, tu = tv.$$

As $\Phi$ is closed under finite intersections, there must exists $T \in \Phi$ such that 
 $$a_{f} = u r_{f},\,\,\, a^{'}_{f}= v r_{f}$$ for all $ f \in T$.

Now suppose that  $f \in T \cap J_{(u,v)}$, then $(u,v) \in f $ so $$(a_{f},a^{'}_{f}) = (u,v) r_{f} \, \in  f S$$  a contradiction to our assumption, hence $(iv)$ implies $(v)$.

$(v)$ implies $(i): $ we will show that the class of left $S$-acts satisfying
Condition $(PWP)$ is axiomatisable by giving explicitly a set of sentences that
axiomatises this class. For any element $t \in S $ with $\mathbf{R}(t,t) \not = \emptyset$, we choose and fix a finite set of
elements $\{(u_{t,1},v_{t,1}) \cdots (u_{t,\,n(t)},v_{t,\,n(t)}) \} $ of 
$\mathbf{R}(t,t)$. We
define sentences $\phi_t$ of $L_s$ as follows:

If $\mathbf{R}(t,t) = \emptyset$ for all $ t \in S$ then $\phi_{t}$ will be

\bigskip

\begin{equation*}
(\forall x)(\forall x')(tx \not = tx');
\end{equation*}

otherwise,  $\phi_t$ is 
\begin{equation*}
(\forall x)(\forall x')\big(tx =tx' \rightarrow (\exists z)(\bigvee
^{n(t)}_{i=1}(x=u_{t,\,i}z \wedge x' = v_{t,\, i}z ))\big) .
\end{equation*}

\vspace{2mm}

Let 
\begin{equation*}
\Sigma_{\mathcal{PWP}} = \Sigma_S \cup \{\phi_{t}:t \in S \}
\end{equation*}

We claim that $\Sigma_{\mathcal{PWP}}$ axiomatises the class $\mathcal{PWP}$.

\vspace{2mm}
Let $A$ be an $S$-act satisfies Condition $(PWP)$. If $\mathbf{R}(t,t) = \emptyset$ and $ta = ta^{'}$ for some $a,a' \in A$ then by Condition  ${(PWP)}$  there exists $u,v \in S$ such that $tu = tv$  a contradiction hence $A \models \phi_{t}$.



If  $ta= ta' $ where $a,a'
\in A $ then using the fact that $A$ satisfies Condition $(PWP)$ there are
elements $s^{\prime },t^{\prime }\in S $ and $c \in A $ such that $%
ts^{\prime } = tt^{\prime }$, $a= s^{\prime }\,c ,\,\,a'= t^{\prime }\,c$. Now $%
(s^{\prime },t^{\prime }) \in \mathbf{R}(t,t)$ and $\mathbf{R}(t,t)\not=\emptyset$. Hence $\phi_{t}$ is  

\begin{equation*}
(\forall x)(\forall x')\big(tx =tx' \rightarrow (\exists z)(\bigvee
^{n(t)}_{i=1}(x=u_{t, \, i}z \wedge x' = v_{t,\, i}z ))\big) 
\end{equation*}

\noindent also $(s^{\prime
},t^{\prime })=(u_{t,\, i},v_{t,\, i})s$ for some $i \in \{1,2, \hdots %
,n(t)\}$ and $s \in S$. Thus  $a= u_{t,\, i}sc,\, a'= v_{t,\, i} sc $. Hence $A \models \phi_{t}$.

\vspace{2mm}

Conversely, let $A$ be a model of $\Sigma_{\mathcal{PWP}}$. If $t \,a = t\,a'$ where $t \in S$ and $a, a' \in
A$,  we cannot have that  $\phi_{t}$ is  $(\forall x)(\forall x') (tx \not = tx')$. It follows that $\mathbf{R}(t,t) \not = \emptyset$ and  $$ f =\{(u_{t,1}, v_{t,1}),\cdots, (u_{t,\,n(t)},v_{t,\,n(t)})\}$$ exists as in $(v)$, and $\phi_t$ is 
\begin{equation*}
(\forall x)(\forall x')\big(tx \,= \,tx' \,\rightarrow \,(\exists \,z) (
\bigvee ^{n(t)}_{i=1}(\,x\,=\,u_{t,\, i}z \wedge x'=  \,v_{t,\, i}
\,z))\big).
\end{equation*}

\vspace{2mm}

Hence there exists an element $c \in A$ with $a=u_{t,\, i} \, c = v_{t,\,
i} \, c$ for some $i \in \{1,2, \hdots, n(t) \} $.  By definition of $%
u_{t,\, i}, v_{t,\, i}$ we have $s\,u_{t,\, i}\, = t v_{t,\, i } $. 
Thus $A$ satisfies Condition $(PWP)$ and so $\Sigma_{\mathcal{PWP}}$ axiomatises the class $\mathcal{PWP}$.
\end{proof}

\begin {cor}\label{cor:axiocondPWPsacts}

The following conditions are equivalent for a monoid $S$:

(i) the class $\mathcal{PWP}$ is axiomatisable;

(ii) for every skeleton $\mathcal{S}=(1,\,t,\,t,\,1)$  over $S$, there exists finitely many replacement skeletons $\mathcal{S}_{1}=(u_1,v_1), \cdots, \mathcal{S}_{q({\mathcal{S}})}=(u_{q(\mathcal{S})},
v_{q(\mathcal{S})})$ of length one such that for any $a,b \in B \in \mathcal{PWP}$ and $ta = tb$ (equivalently,  $(t,a)$ is connected to $ (t,b)$ via a tossing with skeleton $\mathcal{S}$), then $(t,a)$ is connected to $(t,b)$ via a replacement tossing  
with skeleton $\mathcal{S}_{i}$, for some $1 \leq i \leq q(\mathcal{S}).$ 

\end{cor}
\begin{proof} We follow the same argument given in Theorem~\ref{cor:axiocondpsacts}, putting $s=t$ in Remarks ~\ref{rem:length2} and  ~\ref{rem:length1} and using $\mathbf{R}(t,t)$ rather than $\mathbf{R}(s,t)$.
\end{proof}

\section{Examples}\label{sec:examples}

{\bf Example\,(1)}:
Let $G$ be a group with identity $\epsilon$, let $S_1= G$. For the monoid $S_1$ we show the classes $\mathcal{E},\,\mathcal{P},\,\mathcal{EP},\mathcal{W}$ and $\mathcal{PWP}$ of  $S$-acts are axiomatisable.

First note that for any $s,t \in G, \mathbf{r}(s,t)$ is a right ideal, so that $\mathbf{r}(s,t)=G$ and so is finitely generated. Thus $\mathcal{E}$ is axiomatisable. Also $\mathbf{R}(s,t)=(s^{-1}t,\epsilon)G$ is finitely generated, so $\mathcal{P}$ is axiomatisable by Theorem ~\ref{thm:P} (See also \cite{gould}), and by Remark~\ref{rem:EP}, $\mathcal{EP}$ is also axiomatisable. Since every right ideal of $G$ is simply $G$, by Theorem~\ref{thm:condwsact}, $\mathcal{W}$ is axiomatisable.

 Note that $S_1$ being inverse semigroup is absolutely flat by \cite{VF},  $\mathcal{F}$ is axiomatisable and  so is  $\mathcal{WF}.$

\vspace{2mm}

{\bf Example\,(2)}: Let $T$ be an infinite null semigroup. Consider $S_2=T \cup \epsilon$, where $\epsilon$ is an adjoined identity. For  $s \not = t$,  $\mathbf{r}(s,t)= T$ but $T$ is a non-finitely generated (right) ideal of $S_2$.

Moreover $\mathbf{R}(s,t)$ is not always finitely generated. For $s \not = t \not = \epsilon$, $$\mathbf{R}(s,t)= \{(u,v): u,v \not = \epsilon\}.$$

 Suppose on the contrary, $\mathbf{R}(s,t)$ is finitely generated and let  $$\{(u_1,v_1),\cdots, (u_n,v_n)\}$$ be a finite set of generators. 
 
 For any $m \in T$ we have $(m,m) \in \mathbf{R}(s,t)$, so $(m,m)= (u_i,v_i)p$ for some $u_i,v_i \in T$ and $ p \in S_{2}$. If $p= \epsilon$, then $m=u_i=v_i$, if $p \not = \epsilon$, then $m=0$. It follows that $T$ is finite, a contradiction.
 We therefore have $\mathcal{P}$ is not axiomatisable by ~\ref{thm:P}.

 Note that $\mathbf{R}(\epsilon,\epsilon)=(\epsilon,\epsilon)S_2$. Also $\mathbf{R}(s,s)=(\epsilon,\epsilon)\cup(T \times T)$, a similar kind of arguments as given above for $\mathbf{R}(s,t)$ shows that $\mathbf{R}(s,s)$ is not finitely generated. We therefore conclude that $\mathcal{PWP}$ is not axiomatisable.

 Also note that $\mathcal{WF}$ and hence $\mathcal{F}$ are not axiomatisable, see Example $2$ of \cite{bulmanfleminggould} for detail.

{\bf Example\,(3)}: Let $S_{3}$ be a monoid which is a semillatice $\{0,1\}$ of groups $G_{1},\,G_{0}$ with trivial connecting homomorphism. Let $e,\epsilon$ be the identities of $G_{1}$ and $G_{0}$ respectively. If $G_{1}$ is finite then for monoid $S_{3}$ classes $\mathcal{P},\,\mathcal{E},\,\mathcal{W}$ and $\mathcal{PWP}$ are axiomatisable.

\vspace{2mm}

We are supposing that $S_{3}$ is the union of  groups $G_{0}$ and $G_{1}$. Since each (right)
ideal is a union of $G_{0}$ and $G_{1}$, it follows
that $S_{3}$ has only finitely many  ideals. Then every 
ideal of $S_{3}$ is finitely generated, so 
$ \mathbf{r}(s,t)$ is finitely generated.

We will now check that $ \mathbf{R}(s,t)$ is finitely generated for all
$s,t \in S$. Let $ s,t \in G_{0} $. We claim that $\mathbf{R}(s,t)=\mathbf{R}$ where $\mathbf{R}= (e,t^{-1}s)S \cup  (s^{-1}t,e)S\cup\bigcup_{u,v\in  G_{1}\atop su=tv}(u,v)S. $

First note that $ s e = s = \epsilon s = t t^{-1}s $  hence $(e ,t^{-1}s) \in
\mathbf{R}(s,t)$ and so $(e, t^{-1}s)S \subseteq  \mathbf{R}(s,t) $.
With the dual we have $(e,t^{-1}s)S \cup (s^{-1}t,e)S \subseteq \mathbf{R}(s,t) $
then clearly $\mathbf{R}\subseteq  \mathbf{R}(s,t)$.

Conversely, suppose that $ (u,v) \in \mathbf{R}(s,t) $, so that $ su=tv
$. If $u,v\in  G_{1}$, then clearly $(u,v)\in
\mathbf{R}$. If $u\in G_{0}$ then we have that $ u=\epsilon u=s^{-1}su=s^{-1}tv $
so that $(u,v)=(s^{-1}t,e)v$ where $v\in G_{1}$ and
so $(u,v)\in \mathbf{R}$. Together with the dual this tells that
$\mathbf{R}(s,t)\subseteq \mathbf{R}$ and so $\mathbf{R}(s,t)=\mathbf{R}$
as required.
 If $ s \in G_{0} , t \in G_{1}$ we claim
that $ \mathbf{R}(s,t)=\mathbf{R}$ where $ \mathbf{R}=  (s^{-1}t,\epsilon)S \cup(e,t^{-1}s)S \cup \bigcup_{u,v \in  G_{1}\atop su=tv}(u,v)S.$

To see this, notice that $s(s^{-1}t)=(ss^{-1})t=\epsilon t=t\epsilon$  so that $(s^{-1}t,\epsilon)\in \mathbf{R}(s,t)$. Also 
$se=s=es=tt^{-1}s,$ so that $(e,t^{-1}s)\in \mathbf{R}(s,t)$. 
Consequently, $R \subseteq \mathbf{R}(s,t)$.

Conversely, suppose that $(u,v)\in \mathbf{R}(s,t)$. If $u,v\in
 G_{1}$, then clearly $(u,v)\in \mathbf{R}$. If $u\in
G_{0}$ and $v\in G_{1}$, then from $su=tv$ we have that
$\beta=\alpha$ a contradiction.
 On the other hand, if
$u\in G_{1}$ and
$v\in G_{0}$, then $t^{-1}su=t^{-1}tv=ev=v$
so that $(u,v)=(e,t^{-1}s)u\in \mathbf{R}$. Together with the dual this yields that 
$\mathbf{R}(s,t)\subseteq \mathbf{R}$ and so $\mathbf{R}(s,t)=\mathbf{R}$
as required.

Let $s,t \in G_{\beta}$ we claim
that $ \mathbf{R}(s,t)=\mathbf{R}$ where
$ \mathbf{R}=  (\epsilon,\epsilon)S \cup(s^{-1}t,e)S .$

Suppose that $(u,v)\in \mathbf{R}(s,t)$. If $u,v \in G_{0}$ then $su= tv$ implies $u=v$, so that $(u,v)=(\epsilon,\epsilon)u$. The cases where $u \in G_{0},\, v \in G_{1}$ or $u \in G_{1}, v \in G_{0}$ are not possible. Let $u,v \in G_{\beta}$ with $su = tv$ then $u = s^{-1}tv$ so that $(u,v)=(s^{-1}t,e)v$ where $(s^{-1}t,e) \in \mathbf{R}$. Thus $\mathbf{R}(s,t)$ is finitely generated as required.

As we note from Theorems ~\ref{thm:condwsact} and  ~\ref{thm:PWP}, the classes $\mathcal{W}$ and $\mathcal{PWP}$ are axiomatisable if and only if  every ultrapower of $S$ lies in $\mathcal{W}$ and  $\mathcal{PWP}$ respectively. Also  $\mathcal{P}$ implies $\mathcal{W}$, and $\mathcal{P}$ implies $ \mathcal{PWP}$, so by using Lemma ~\ref{lem:UV} we conclude that $\mathcal{W}$ and $\mathcal{PWP}$ are axiomatisable.

\vspace{2mm}

{\bf Example\,(4)}: Consider the monoid $S_{4}= (\mathbb{Z} \times \mathbb{Z}) \cup (\epsilon, \epsilon)$ where the $(\epsilon, \epsilon)$ is the adjoined identity element with binary operation given by $$(a,b)(c,d)= (a-b +max \{b,c\}, d-c + max\{b,c\})$$ of $S_{3}$. Note that $\mathbf{r}\big((a,b),(c,d)\big)$ which is a right ideal of $S_{4}$ for any $(a,b),(c,d)  \in  S_{4}$ is finitely generated. Hence $\mathcal{E}$ is axiomatisable but $\mathbf{R}\big((a,b),(c,d)\big)$ is not finitely generated, so that the class $\mathcal{P}$ is not axiomatisable. Moreover $S_4$ is inverse semigroup, and hence absolutely flat, so that $\mathcal{F}$ and $\mathcal{WF}$ are axiomatisable. For details we refer reader to  \cite{gould:tartu}.

\vspace{2mm}

 As we know that $\mathcal{P} $ implies $ \mathcal{W}$, and by using Lemma~\ref{lem:UV} if $\mathcal{P}$ is axiomatisable for a monoid $S$ then so is $\mathcal{W}$. But the converse of this statement is not true, note that $S_{4}$ satisfies Condition $(iv)$ of Theorem~\ref{thm:condwsact} as each ideal of $S_{4}$ is principal, see \cite{gould:tartu} and intersection of two ideals is again an ideal, so  is principal.  Thus $\mathcal{W}$ is axiomatisable.

 {\bf Example\,(5)}: Consider $S_{5}= (\mathbb{N},min) \cup \epsilon$, we show that for $S_5$, $\mathbf{R}(s,t)$ is not finitely generated, but $\mathbf{r}(s,t)$ is finitely generated.

  Let $s <t$ then $ \mathbf{r}(s,t)=\mathbf{r}(t,s)=\{1,2,\cdots, s\}= sS $. Hence $\mathcal{E}$ is axiomatisable.

Let $s \not = t$ with $sa= ta$ for some $ a \in {A} \in \mathcal{EP}$, $a= ua''=va''$ for some $a'' \in A$ and $su=tv.$

 If $s \leq t$ with $sa = ta$ we could have $a=sa=sa$ where $ss= ts$ therefore we could take $\{s,s\} \in f$.

Or   $u >s, \,\, us = s$ \,\,so for \,\,$tv= s$ need $v=s$ so that we have $\{(t,s ):t >s\} \in f$.

$$f=\{(1,1),\cdots,(s,s),(s+1,s),(s+2,s),\cdots ,(t,s)\}= (s,s)S \cup (\epsilon,s)S $$ a finite subset of $\mathbf{R}(s,t)$. Thus condition $(iv)$ of Theorem~\ref{thm:(EP)} to axiomatise the class $\mathcal{EP}$ holds.

\bigskip

Note that for $s \leq t,\, \mathbf{R}(s,t)=\{(1,1),\cdots,(s,,s),(s+1,s),\cdots \}= (s,s)S \cup (\epsilon,s)S$ so that  $\mathbf{R}(s,t)$ is finitely generated if $s \not = t$.

\bigskip

We can check that $\mathbf{R}(s,s)$ is not finitely generated.

\bigskip

Suppose on contrary, $\mathbf{R}(1,1)$ is finitely generated.
If $(u_1,v_1),\cdots,(u_n,v_n)$  is a finite set of generators of $\mathbf{R}(1,1)$, let $(\epsilon,m)=(u_{i},v_{i})t$ for some $i$ therefore $u_i= \epsilon,t=\epsilon$ therefore $m=v_{i}$. Hence $\mathcal{P}$ and $\mathcal{PWP}$ are not axiomatisable.

\begin{rem} We make the follwoing connections between the axiomatisability conditions of the following classes of $S$-acts, some of them are still unknown.

\[\begin{array}{rclcrcl}
  
                 \mathcal{P} & \Rightarrow& \mathcal{EP}       &  & Remark~\ref{rem:EP}\\ 
               \mathcal{P} & \not\Leftarrow & \mathcal{EP}    &  & Example\, {5} &  \\
               \mathcal{E}& \Rightarrow  &   \mathcal{EP}                      && Unknown & \\                    
     					\mathcal{E} & \not\Leftarrow & \mathcal{EP}    &   & Unknown&   \\
              \mathcal{P} & \Rightarrow& \mathcal{W}         & & Lemma ~\ref{lem:UV}\\
              \mathcal{P} & \not\Leftarrow&  \mathcal{W}          &  &  Example\,{4}& \\
              \mathcal{P} & \Rightarrow&  \mathcal{PWP}          &  & Lemma ~\ref{lem:UV}\\
              \mathcal{P} & \not\Leftarrow&  \mathcal{PWP}          & & Unknown & \\
              \mathcal{F} & \Rightarrow&  \mathcal{WF}          &  & Lemma ~\ref{lem:UV}\\
               \mathcal{F} & \not\Leftarrow&  \mathcal{WF}          &  & Unknown& \\
\end {array} \]

\end{rem}

\section{Some Open Problems}

\bigskip

\subsection{Axiomatisability of Condition (WP)}
We first describe the Condition $(WP)$;

\vspace{2mm}

A left $S$-act satisfies Condition $(WP)$ if for every pullback diagram $(P,(p_1,p_2))$ 
of the pair $(f,f)$ where each $f :I \to S$ is a $S$-homomorphism, the corresponding map $\gamma$ is surjective, for some right ideal $I$  of $S$ 

or  equivalently

A left $S$-act $A$ satisfies condition $(WP)$ if and only if for all $S$-homomorphism $f:(sS \cup tS)_{S} \to S_{S}$ where $s,t \in S$ and all $a,a' \in A$  if $(s)f a = (t)f a'$  then there exists $a'' \in A$, $u,v \in S$, $s' , t' \in \{ s,t \}$ such that $(s' u)f=(t'v)f$, $s \otimes a = s' u \otimes a'' $, and $t \otimes a' = t'v \otimes a''$ in $(sS \cup tS)_{S} \otimes _{S}A$.

We aim to axiomatise the class of left $S$-acts satisfying Condition $(WP)$.

It is still an open problem to determine conditions on a monoid $S$ such that classes $\mathcal{WF},\,\mathcal{F}$ are coincide,  also to characterise those conditions on a monoid $S$ such that $\mathcal{WF}$ is axiomatisable  and $\mathcal{F}$ not, or $\mathcal{PWF}$ is axiomatisable but $\mathcal{WF}$  not.

\end{document}